\documentclass{tran-l}

\newtheorem{theorem}{Theorem}[section]
\newtheorem{lemma}[theorem]{Lemma}
\newtheorem{proposition}[theorem]{Proposition}

\theoremstyle{definition}

\theoremstyle{remark}

\numberwithin{equation}{section}

\newcommand{\QED}{\qed}
\newcommand{\n}{\par\noindent}
\newcommand{\sn}{\par\smallskip\noindent}
\newcommand{\mn}{\par\medskip\noindent}
\newcommand{\bn}{\par\bigskip\noindent}
\newcommand{\pars}{\par\smallskip}
\newcommand{\parm}{\par\medskip}

\newcommand{\bfind}[1]{{\bf #1}}

\newcommand{\sep}{^{\rm sep}}
\newcommand{\chara}{\mbox{\rm char}\,}
\newcommand{\trdeg}{\mbox{\rm trdeg}\,}
\newcommand{\Gal}{\mbox{\rm Gal}\,}
\newcommand{\rr}{\mbox{\rm rr}\,}
\newcommand{\ic}{\mbox{\rm IC}\,}

\newcommand{\cal}{\mathcal}
\newcommand{\eu}{\mathfrak}
\newcommand{\N}{\mathbb N}

\begin{document}

\title{Corrections and Notes for the Paper\\
``Value groups, residue fields and bad places of rational
function fields''}

\author{Anna Blaszczok}
\address{Institute of Mathematics,
University of Silesia,
Bankowa 14,
40-007 Katowice,
Poland }
\email{ablaszczok@us.edu.pl}

\author{Franz-Viktor Kuhlmann}
\address{Department of Mathematics and Statistics,
University of Saskatchewan,
106 Wiggins Road,
Saskatoon, Saskatchewan, Canada S7N 5E6}
\email{fvk@math.usask.ca}

\subjclass[2000]{Primary 12J10; Secondary 12J15, 16W60}
\date{22. 2. 2014}
\thanks{Part of this paper was prepared while the first author was a
guest of the Department of Mathematics and Statistics of the University
of Saskatchewan. She gratefully acknowledges their hospitality and
support.\\ The second author was partially supported by a Canadian NSERC
grant. He also gratefully acknowledges the hospitality of the
Institute of Mathematics of the University of Silesia in Katowice, Poland.}

\begin{abstract}
We correct mistakes in the paper \cite{[K1]} and report on recent new
developments which settle cases left open in that paper.
\end{abstract}

\maketitle
\markboth{BLASZCZOK AND KUHLMANN}{CORRECTIONS AND NOTES FOR: VALUE
GROUPS, RESIDUE FIELDS AND BAD PLACES}

%
%
\section{Introduction}
The main theorem of \cite{[K1]} (Theorem~1.6) describes all extensions
of a valuation $v$ of a base field $K$ to a rational function field
\[
F\>=\>K(x_1,\ldots,x_n)
\]
of transcendence degree $n\geq 1$. There was an omission in the last
part, B4), of the theorem; we will state a corrected version of the
full theorem in Section~\ref{sectext}. The case B4), where the desired
value group and residue field extensions $vF|vK$ and $Fv|Kv$ satisfy
that
\begin{equation}                            \label{caseB4}
vF/vK\mbox{ \ and \ } Fv|Kv\mbox{ \ are finite,}
\end{equation}
was the only one that was not completely understood. In particular, a
full converse of the assertion of the theorem in this case is not known.
This is due to deep open problems in the theory of immediate extensions
of valued fields with residue fields of positive characteristic.

In the paragraph following Theorem~1.7 in \cite{[K1]} (which presents a
partial converse of Theorem~1.6), the second author wrote that a full
converse can be given if $Kv$ has characteristic 0 or $(K,v)$ is a
Kaplansky field. This claim, not proven in \cite{[K1]}, is correct, and
we will show an even stronger result in Theorem~\ref{tame_rat_ff} below.
But he also claimed that the reason was that for such valued fields, if
$(\tilde{K},v)$ admits an immediate extension of transcendence degree
$n$, then so does $(K,v)$. This statement is false, as we will show in
Section~\ref{sectext}.

Working from the other direction on closing the gap, we will show in
Theorem~\ref{gen1.1} that the conditions of B4) can be slightly relaxed.
To facilitate a quicker assessment, the proofs of Theorems~\ref{gen1.1}
and~\ref{tame_rat_ff} are shifted to Section~\ref{sectproofs}.
Throughout Sections~\ref{sectext} and~\ref{sectproofs}, we take $F|K$ as
above and assume that $v$ is nontrivial on $K$ (since otherwise there is
no extension of $v$ to $F$ satisfying conditions (\ref{caseB4})).

\parm
In Section~5 of \cite{[K1]} the second author defined homogeneous
sequences as a tool for the computation of the value group and residue
field of a simple extension $(K(x),v)$ of $(K,v)$. Due to a last minute
change in the definition, the crucial Lemma~5.1 of \cite{[K1]} became
false. In Section~\ref{secthom} of this paper we give the correct
definition and show that with the help of it, all results of Section~5
of \cite{[K1]} can be proved. For the convenience of the reader, we will
include all results of that section in the present paper, but will omit
those proofs that do not need to be changed.

\parm
Finally, we include a list of corrections for some typing
errors in \cite{[K1]} in Section~\ref{sectcor} of the present paper.

\parm
We take over the notions and notations from \cite{[K1]}.

%
%
\section{Extensions of valuations to rational function
fields}                                     \label{sectext}
Here is a version of Theorem~1.6 of \cite{[K1]} with a small correction
which we will discuss afterwards:
\begin{theorem}                             \label{MTsevvar}
Let $(K,v)$ be any valued field, $n,\rho,\tau$ non-negative integers,
$n\geq 1$, $\Gamma\ne\{0\}$ an ordered abelian group extension of $vK$
such that $\Gamma/vK$ is of rational rank $\rho$, and $k|Kv$ a field
extension of transcendence degree $\tau$.
\sn
{\bf Part A.} \ Suppose that $n>\rho+\tau$ and that
\sn
A1) \ $\Gamma/vK$ and $k|Kv$ are countably generated,
\n
A2) \ $\Gamma/vK$ or $k|Kv$ is infinite.
\sn
Then there is an extension of $v$ to the rational
function field $K(x_1,\ldots,x_n)$ in $n$ variables such that
\begin{equation}                            \label{rfwvgarf}
vK(x_1,\ldots,x_n)\>=\>\Gamma\;\;\mbox{\ \ and\ \ }\;\;
K(x_1,\ldots,x_n)v\>=\>k\;.
\end{equation}
\mn
{\bf Part B.} \ Suppose that $n\geq\rho+\tau$ and that
\sn
B1) \ $\Gamma/vK$ and $k|Kv$ are finitely generated,
\n
B2) \ if $v$ is trivial on $K$, $n=\rho+\tau$ and $\rho=1$, then $k$
is a simple algebraic extension of a rational function field in $\tau$
variables over $Kv$ (or of $Kv$ itself if $\tau=0$), or a rational
function field in one variable over a finitely generated field extension
of $Kv$ of transcendence degree $\tau-1$,
\n
B3) \ if $n=\tau$, then $k$ is a rational function field
in one variable over a finitely generated field extension of $Kv$ of
transcendence degree $\tau-1$,
\n
B4) \ if $\rho=0=\tau$, then there is an immediate extension $(L|K,v)$
which either is separable-algebraic such that the extension
$(L^h|K^h,v)$ of their respective henselizations is infinite, or is of
transcendence degree at least $n$.
\sn
Then again there is an extension of $v$ to
$K(x_1,\ldots,x_n)$ such that (\ref{rfwvgarf}) holds.
\end{theorem}

\pars
In the original version of assumption B4), it was
stated that the existence of an infinite immediate separable-algebraic
extension $(L|K,v)$ is sufficient for the assertion of Theorem~1.6 to
hold. But it has to be required in addition that also the extension
$(L^h|K^h,v)$ of their respective henselizations is infinite. With this
additional assumption the proof of the theorem as presented in
\cite{[K1]} remains unchanged. Note that the henselization is always a
separable-algebraic extension, but the assumption that it is infinite is
not enough for our theorem.

As an example, take
%
%
the Laurent series field $k((t))$ with its $t$-adic valuation. This is a
maximal field, that is, it does not admit any proper immediate
extensions. Take a transcendence basis $T$ of $k((t))|k(t)$. Then
$k((t))$ is the henselization of $K:=k(t)(T)$, an infinite
separable-algebraic and immediate extension of $K$. But $K$ does not
admit any extension of the valuation to $F$ such that $vF=vK$ and
$Fv=Kv$, since then the henselization $F^h$ would be a proper immediate
extension of $K^h=k((t))$.

\pars
It suffices to assume the existence of an infinite immediate
separable-algebraic extension $(L|K^h,v)$ because then $L$, being an
algebraic extension of a henselian field is henselian itself, $(L|K,v)$
is also an immediate separable-algebraic extension, and $L^h=L$ is an
infinite extension of $K^h$.
In order to analyze the situation further, we cite the following
theorem, which is a special case of Theorem 1.1 in the paper
\cite{[BK1]}. This paper is a significantly extended version of the
paper [KU5] cited in~\cite{[K1]}.

\begin{theorem}                       \label{MTai}
Take an algebraic extension $(L|K,v)$ and assume that
the extension $(L^h|K^h,v)$ of their henselizations contains an
infinite separable-algebraic subextension.
Then each maximal immediate extension
of $(L,v)$ has infinite transcendence degree over $L$.
\end{theorem}

If $(L|K^h,v)$ is an infinite immediate separable-algebraic extension,
then by this theorem, $(L,v)$ admits an immediate extension $(M|L,v)$ of
infinite transcendence degree. Since $(L|K,v)$ is immediate too, it
follows that also $(M|K,v)$ is immediate of infinite transcendence
degree. This shows that the above assumption actually implies the other
assumption stated in case B4) of Theorem 1.6 of~\cite{[K1]}: the
existence of an immediate extension of transcendence degree $n$.

The main problem with case (\ref{caseB4}) is that we do not know a full
converse of the corresponding assertion in Theorem~1.6 of~\cite{[K1]}.
In fact, Theorem~1.7 of \cite{[K1]} states only this much:
\sn
{\it Let $n\geq 1$ and $v$ be a valuation on the rational function field
$F=K(x_1,\ldots,x_n)$. Set $\rho=\rr vF/vK$ and $\tau=\trdeg Fv|Kv$.
Then $n\geq\rho+\tau$, $vF/vK$ is countable, and $Fv|Kv$ is
countably generated.

If $n=\rho+\tau$, then $vF/vK$ is finitely generated and $Fv|Kv$ is a
finitely generated field extension. Assertions B2) and B3) of
Theorem~\ref{MTsevvar} hold for $k=Fv$, and if $\rho=0= \tau$, then
there is an immediate extension of $(\tilde{K},v)$ of transcendence
degree~$n$ (for any extension of $v$ from $K$ to $\tilde{K}$).}
\sn
Hence if an extension of $v$ from $K$ to $F$ with properties
(\ref{caseB4}) exists, then the algebraic closure $\tilde{K}$ of $K$
admits an immediate extension of transcendence degree $n$, namely,
$(\tilde{K}.F| \tilde{K},v)$. But we would like to know something about
$K$, not only about its algebraic closure. Is it true that the existence
of an extension with properties (\ref{caseB4}) always implies the
existence of an immediate extension of transcendence degree~$n$? The
following generalization of the assertion of Theorem~\ref{MTsevvar} for
the case (\ref{caseB4}), which we will prove in
Section~\ref{sectproofs}, casts some doubt on this. However, we do not
know of any example where the condition of the theorem is met but an
immediate extension of transcendence degree $n$ does not exist.


\begin{theorem}                             \label{gen1.1}
Take a finite ordered abelian group
extension $\Gamma$ of $vK$ and a finite extension $k$ of $Kv$. Assume
that the henselization $K^h$ admits an infinite separable-algebraic
extension $(L|K^h,v)$ with $vL\subseteq\Gamma$ and $Lv\subseteq k$. Then
there is an extension of $v$ from $K$ to $F$ such that $vF=\Gamma$ and
$Fv=k$.
\end{theorem}

Theorem~\ref{MTai} can also be used to show that the statement from the
paragraph following Theorem~1.7 in \cite{[K1]} which we indicated in the
introduction is false: even if $(\tilde{K},v)$ admits an immediate
extension of transcendence degree $n$, the same is not necessarily true
for $(K,v)$. Indeed, take any maximal valued field $(K,v)$ that is not
separable-algebraically closed or real closed. Certainly, there is an
abundance of valued fields with residue characteristic 0 and of
Kaplansky fields that satisfy the stated conditions. Then $(K,v)$ is
henselian and its separable-algebraic closure is an infinite extension.
Hence by Theorem~\ref{MTai}, each maximal immediate extension of
$(\tilde{K},v)$ has infinite transcendence degree over $\tilde{K}$,
while $(K,v)$ itself does not have any proper immediate extensions.

\pars
While the question about a full converse is still open, the following
theorem settles the problem for a large class of valued fields which
includes valued fields of residue characteristic 0, Kaplansky fields and
tame fields. Note that the definition of the \bfind{implicit constant
field}, which in \cite{[K1]} was given for valued rational function
fields in one variable, can be generalized without problems; we take
$\ic(F|K,v)$ to be the relative algebraic closure of $K$ in a fixed
henselization of $(F,v)$.

\begin{theorem}                        \label{tame_rat_ff}
Take $p$ to be the
characteristic exponent of $Kv$. Assume that $vK$ is $p$-divisible and
$Kv$ is perfect. Further, take an ordered abelian group extension
$\Gamma$ of $vK$ such that $\Gamma/vK$ is a torsion group, and an
algebraic extension $k$ of $Kv$. Then there is an extension of $v$ from
$K$ to $F$ with $vF=\Gamma$ and $Fv=k$ if and only if at least one of
the two extensions $\Gamma|vK$ and $k|Kv$ is infinite or $(K,v)$ admits
an immediate extension of transcendence degree $n$. In this case, $\ic
(F|K,v)$ is a separable-algebraic extension $(L,v)$ of $(K,v)$ with
$vL=\Gamma$ and $Lv=k$, and it is infinite over the henselization of $K$
if $\Gamma/vK$ or $k|Kv$ is infinite.
\end{theorem}

For the proof, which we will give in Section~\ref{sectproofs}, we will
use the following powerful theorem from \cite{[B]} (see also
\cite{[BK2]}):

\begin{theorem}                          \label{almK_inftrdeg}
Take a valued field $(K,v)$ of positive residue characteristic $p$,
with $p$-divisible value group and perfect residue field.
\sn
1) \ If $(K,v)$ admits a maximal immediate extension of finite
transcendence degree, then the maximal immediate extension
of $K$ is unique up to valuation preserving isomorphism.
\sn
2) \ If $(K,v)$ admits a finite separable-algebraic extension $(K',v)$
such that the valuation $v$ extends in a unique way from $K$ to $K'$
and $(K'|K,v)$ has nontrivial defect, then every maximal immediate
extension of $(K,v)$ is of infinite transcendence degree \mbox{over $K$.}
\end{theorem}

%
%
%
\section{Homogeneous sequences}              \label{secthom}
In Section~5.1 of \cite{[K1]}, the notion of ``homogeneous
approximation'' was introduced. But the definition was incorrect, with
the consequence that Lemma~5.1 of~\cite{[K1]} does not hold for this
definition. The correct definition is as follows.

Let $(K,v)$ be any valued field and $a,b$ elements in some valued field
extension $(L,v)$ of $(K,v)$. We will say that $a$ is \bfind{strongly
homogeneous over $(K,v)$} if $a\in K\sep\setminus K$, the extension of
$v$ from $K$ to $K(a)$ is unique, and
\[
va\;=\;\mbox{\rm kras}(a,K)\>:=\> \max\{v(\tau a-\sigma a)\mid
\sigma,\tau\in \Gal K \mbox{\ \ and\ \ } \tau a\ne \sigma a\}
\;\in\>v\tilde{K}\;.
\]
%
%
We call $a\in L$ a \bfind{homogeneous approximation of $b$ over $K$} if
there is some $d\in K$ such that $a-d$ is strongly homogeneous over $K$
and $v(b-a)>v(b-d)\geq vb$. It then follows that $va=vb$ and
$v(a-d)=v(b-d)$. With this definition, Lemma~5.1 of \cite{[K1]} holds:

\begin{lemma}                               \label{krlkra}
If $a\in L$ is a homogeneous approximation of $b$ then $a$ lies in the
henselization of $K(b)$ w.r.t.\ every extension of the valuation $v$
from $K(a,b)$ to $\widetilde{K(b)}$.
\end{lemma}
\begin{proof}
From Lemma~2.21 of \cite{[K1]} we obtain that $a-d$ and hence also
$a$ lies in the henselization of $K(b-d)=K(b)$ w.r.t.\ every extension
of the valuation $v$ from $K(a,b)$ to $\widetilde{K(b)}$.
\end{proof}

Lemmas~5.2 and~5.3 of~\cite{[K1]} remain unchanged:
\begin{lemma}                               \label{homup}
Let $(K',v)$ be any henselian extension field of $(K,v)$ such that
$a\notin K'$. If $a$ is homogeneous over $(K,v)$, then it is also
homogeneous over $(K',v)$, and $\mbox{\rm kras}(a,K)=\mbox{\rm kras}
(a,K')$. If $a$ is strongly homogeneous over $(K,v)$, then it is also
strongly homogeneous over $(K',v)$.
\end{lemma}

\begin{lemma}                              \label{lstronghom}
Suppose that $a\in\tilde{K}$ and that there is some extension of $v$
from $K$ to $K(a)$ such that if {\rm e} is the least positive integer
for which ${\rm e}va\in vK$, then
\sn
a) \ {\rm e} is not divisible by $\chara Kv$,\n
b) \ there exists some $c\in K$ such that $vca^{\rm e}=0$, $ca^{\rm e}v$
is separable-algebraic over $Kv$, and the degree of $ca^{\rm e}$ over
$K$ is equal to the degree {\rm f} of $ca^{\rm e}v$ over $Kv$.
\sn
Then $[K(a):K]={\rm ef}$ and if $a\notin K$, then $a$ is strongly
homogeneous over $(K,v)$.
\end{lemma}

Lemma~5.4 of~\cite{[K1]} should read:
\begin{lemma}                               \label{exkrap}
Assume that $b$ is an element in some algebraically closed valued field
extension $(L,v)$ of $(K,v)$. Suppose that there is some ${\rm e}\in\N$
not divisible by $\chara Kv$, and some $c\in K$ such that $vcb^{\rm e}
=0$ and $cb^{\rm e}v$ is separable-algebraic over $Kv$. If the
smallest possible ${\rm e}\in\N$ is bigger than $1$ or if $cb^{\rm
e}v\notin Kv$, then we can find $a\in L$, strongly homogeneous over $K$
and such that $v(b-a)>vb$. In particular, $a$ is a homogeneous
approximation of $b$ over $K$.
\end{lemma}
\begin{proof}
Take a monic polynomial $g$ over $K$ with $v$-integral coefficients
whose reduction modulo $v$ is the minimal polynomial of $cb^{\rm e}v$
over $Kv$. Then let $a_0\in\tilde{K}$ be the root of $g$ whose residue
is $cb^{\rm e}v$. The degree of $a_0$ over $K$ is the same as that of
$cb^{\rm e}v$ over $Kv$. We have that $v(\frac{a_0}{cb^{\rm e}}-1)>0$.
So there exists $a_1\in \tilde{K}$ with residue $1$ and such that
$a_1^{\rm e}= \frac{a_0}{cb^{\rm e}}$. Then for $a:=a_1 b$, we find that
$v(a-b)=vb+ v(a_1-1)>vb$ and $ca^{\rm e}=a_0\,$. It follows that $va=vb$
and $ca^{\rm e}v=cb^{\rm e}v$. By the foregoing lemma, this shows that
$a$ is strongly homogeneous over $K$.
\end{proof}

The definition of homogeneous sequences and Lemma~5.5 of~\cite{[K1]}
remain unchanged:

Let $(K(x)|K,v)$ be any extension of valued fields. We fix an extension
of $v$ to $\widetilde{K(x)}$.
Let $S$ be an initial segment of $\N$, that is, $S=\N$ or
$S=\{1,\ldots,n\}$ for some $n\in\N$ or $S=\emptyset$. A sequence
\[{\eu S}\;:=\;(a_i)_{i\in S}\]
of elements in $\tilde{K}$  will be called a \bfind{homogeneous sequence
for $x$} if the following conditions are satisfied for all $i\in S$
(where we set $a_0:=0$):
\sn
{\bf (HS)} \ \ $a_i-a_{i-1}$ is a homogeneous approximation of
$x-a_{i-1}$ over $K(a_0,\ldots,a_{i-1})$.
\sn
Recall that then by the definition of ``strongly homogeneous'',
$a_i\notin K(a_0,\ldots,a_{i-1})^h$. We call $S$ the \bfind{support} of
the sequence $\eu S$. We set
\[
K_{\eu S}\;:=\;K(a_i\mid i\in S)\;.
\]
If ${\eu S}$ is the empty sequence, then $K_{\eu S}=K$.

\begin{lemma}                               \label{kspcs}
If $i,j\in S$ with $1\leq i<j$, then
\begin{equation}                            \label{pcs}
v(x-a_j)\> >\>v(x-a_i)\>=\>v(a_{i+1}-a_i)\;.
\end{equation}
If $S=\N$ then $(a_i)_{i\in S}$ is a pseudo Cauchy sequence in
$K_{\eu S}$ with pseudo limit $x$.
\end{lemma}

Lemma~5.6 of~\cite{[K1]} now reads:
\begin{lemma}                               \label{xx}
Take $x,x'\in L$.\n
1) \ If $a\in L$ is a homogeneous approximation of $x$ over $K$ and if
$v(x-x')\geq v(x-a)$, then $a$ is also a homogeneous approximation of
$x'$ over $K$.
\sn
2) \ Assume that $(a_i)_{i\in S}$ is a homogeneous sequence for $x$ over
$K$. If $v(x-x')>v(x-a_k)$ for all $k\in S$, then $(a_i)_{i\in S}$ is
also a homogeneous sequence for $x'$ over $K$.

In particular, for each $k\in S$ such that $k>1$, $(a_i)_{i<k}$ is a
homogeneous sequence for $a_k$ over $K$.
\end{lemma}
\begin{proof}
Only the proof of part 1) changes:
\sn
Suppose that $a$ is a homogeneous approximation of $x$ over $K$,
with $v(x-a)>v(x-d)\geq vx$ and $a-d$ strongly homogeneous over $K$. If
in addition $v(x-x')\geq v(x-a)>v(x-d)$, then $v(x'-d)=\min\{v(x-x'),
v(x-d)\} =v(x-d)$ and $v(x'-a)\geq \min\{v(x-x'),v(x-a)\}\geq v(x-a)>
v(x-d)=v(x'-d)$. Furthermore, $v(x-x')>v(x-d)\geq vx$, hence $vx=vx'$
and $v(x'-d)\geq vx'$. This yields the first assertion.
\end{proof}

The remainder of Section~5.2 of~\cite{[K1]} remains unchanged, but for
the convenience of the reader, we repeat the results here without proof:

\begin{lemma}                               \label{KSinL}
Assume that $(a_i)_{i\in S}$ is a homogeneous sequence for $x$ over $K$.
Then
\begin{equation}                            \label{KSin}
K_{\eu S} \>\subset\> K(x)^h \;.
\end{equation}
For every $n\in S$, $a_1,\ldots,a_n\in K(a_n)^h$. If
$S=\{1,\ldots,n\}$, then
\begin{equation}                            \label{in}
K_{\eu S}^h \>=\> K(a_n)^h\;.
\end{equation}
\end{lemma}

\begin{proposition}                         \label{SNpure}
Assume that ${\eu S}=(a_i)_{i\in S}$ is a homogeneous sequence for
$x$ over $K$ with support $S=\N$. Then $(a_i)_{i\in\N}$ is a pseudo
Cauchy sequence of transcendental type in $(K_{\eu S},v)$ with pseudo
limit $x$, and $(K_{\eu S}(x)|K_{\eu S},v)$ is immediate and pure.
\end{proposition}

This proposition leads to the following definition. A homogeneous
sequence $\eu S$ for $x$ over $K$ will be called \bfind{(weakly) pure
homogeneous sequence} if $(K_{\eu S}(x)|K_{\eu S},v)$ is (weakly) pure
in $x$. Hence if $S=\N$, then $\eu S$ is always a pure homogeneous
sequence. The empty sequence is a (weakly) pure homogeneous sequence for
$x$ over $K$ if and only if already $(K(x)|K,v)$ is (weakly) pure in
$x$.

\begin{theorem}                             \label{thIC}
Suppose that $\eu S$ is a (weakly) pure homogeneous sequence for $x$
over $K$. Then
\[K_{\eu S}^h \;=\; \ic (K(x)|K,v)\;.\]
Further, $K_{\eu S}v$ is the relative algebraic closure of $Kv$ in
$K(x)v$, and the torsion subgroup of $vK(x)/vK_{\eu S}$ is finite.
If $\eu S$ is pure, then $vK_{\eu S}$ is the relative divisible
closure of $vK$ in $vK(x)$.
\end{theorem}

\pars
In Section~5.2 of~\cite{[K1]}, Proposition~5.10 remains unchanged:
\begin{proposition}                         \label{hit}
Suppose that $(K,v)$ is henselian. If $a$ is homogeneous over $(K,v)$,
then $(K(a)|K,v)$ is a tame extension. If $\eu S$ is a homogeneous
sequence over $(K,v)$, then $K_{\eu S}$ is a tame extension of $K$.
\end{proposition}
\n
It can also be shown that if $va\,=\,\mbox{\rm kras}(a,K)$ {\it and} $a$
is separable over $K$, then $a$ satisfies the conditions of
Lemma~\ref{lstronghom}. In~\cite{[K1]}, the separability condition was
forgotten.

\pars
Because of the change in the definition, the proof of Proposition~5.11
of~\cite{[K1]} changes significantly. Here is the proposition with its
new proof:

\begin{proposition}
An element $b\in \tilde{K}$ belongs to a tame extension of the henselian
field $(K,v)$ if and only if there is a finite homogeneous sequence
$a_1,\ldots,a_k$ for $b$ over $(K,v)$ such that $b\in K(a_k)$.
\end{proposition}
\begin{proof}
Suppose that such a sequence exists. By Proposition~5.10 of \cite{[K1]},
$K_{\eu S}$ is a tame extension of $K$. Since $b\in K(a_k)\subseteq
K_{\eu S}$, it contains $b$.

For the converse, let $b$ be an element in some tame extension of
$(K,v)$. Since $K(b)|K$ is finite, also the extensions $vK(b)|vK$ and
$K(b)v|Kv$ are finite. Take $\eta_i\in K(b)$ with $\eta_1=1$ such that
$v\eta_i\,$, $1\leq i\leq\ell$, belong to distinct cosets modulo $vK$.
Further, take $\vartheta_j \in {\cal O}_{K(b)}$ with $\vartheta_1 =1$
such that $\vartheta_j v$, $1\leq j\leq m$, are \mbox{$Kv$-linearly}
independent. Then by Lemma~2.8 of \cite{[K1]}, the elements $\eta_i
\vartheta_j\,$, $1\leq i\leq \ell$, $1\leq j\leq m$, are $K$-linearly
independent. Since $(K(b)|K,v)$ is tame, we have that $[K(b):K]=\ell m$,
so these elements form a basis of $K(b)|K$. Now we write
\[
b\>=\>\sum_{i,j} c_{ij} \eta_i\vartheta_j
\]
with $c_{ij}\in K$. Again by Lemma~2.8 of \cite{[K1]},
%
\[
vb\>=\>v\sum_{i,j} c_{ij} \eta_i\vartheta_j\>=\>
\min_{i,j}\, v c_{ij}\eta_i\vartheta_j\>=\>
\min_{i,j}\, (vc_{ij}\,+\,v\eta_i)\;.
\]
If $c_{11}\eta_1\vartheta_1=c_{11}\in K$ happens to be the unique
summand of minimal value, then we set $d=c_{1,1}$; otherwise, we set
$d=0$.

Choose $i_0$ such that $v(b-d)$ is in the coset of $v\eta_{i_0}$.
If $vc_{i_1j_1} \eta_{i_1}=vc_{i_2j_2} \eta_{i_2}$ then $i_1=i_2$
since the values $v\eta_i$ are in distinct cosets modulo $vK$.
So we can list the summands of minimal value as $c_{i_0j_r}
\eta_{i_0}\vartheta_{j_r}$, $1\leq r\leq t$, for some
$t\leq m$. We obtain that
\begin{equation}                            \label{b-d-a}
v\left(b-d-\sum_{r=1}^{t}c_{i_0j_r}\eta_{i_0}\vartheta_{j_r}\right)
\>>\>v(b-d)\>.
\end{equation}

Take e to be the least positive integer such that $ev(b-d)\in vK$.
Choose $c\in K$ such that $vc(b-d)^{\rm e}=0$. Since $Kv$ is perfect,
$c(b-d)^{\rm e}v$ is separable-algebraic over $Kv$. If $i_0>1$, then
$v(b-d)\notin vK$. Hence ${\rm e}>1$ and since $(K,v)$ is a tame field,
e is not divisible by $\chara Kv$. If $i_0=1$, then ${\rm e}=1$ and
$\eta_{i_0}=1$, and in view of (\ref{b-d-a}),
\[
c(b-d)v\>=\>\sum_{r=1}^{t}(cc_{i_0j_r}v)\cdot \vartheta_{j_r}v\>.
\]
This is not in $Kv$ since by our choice of $d$, some $j_k>1$ must appear
in the sum and the residues $\vartheta_{j_r}v$, $1\leq r\leq t$, are
linearly independent over $Kv$.

We conclude that $b-d$ satisfies the assumptions of Lemma~\ref{exkrap}.
Hence there is an element $a\in\tilde{K}$, strongly homogeneous over $K$
and such that $v(b-d-a)>v(b-d)\geq vb$. We set $a_1:=a+d$ to obtain that
$a_1$ is a homogeneous approximation of $b$ over $K$. By the foregoing
proposition, $K(a_1)$ is a tame extension of $K$ and therefore, by the
general facts we have noted following the definition of tame extensions
in \cite{[K1]}, $K(a_1, b-a_1)$ is a tame extension of $K(a_1)$.

We repeat the above construction, replacing $b$ by $b-a_1\,$. By
induction, we build a homogeneous sequence for $b$ over $K$.
It cannot be infinite since $b$ is algebraic over $K$ (cf.\
Proposition~\ref{SNpure}). Hence it stops with some element
$a_k\,$. Our construction shows that this can only happen if $b\in
K(a_1,\ldots,a_k)$, which by Lemma~\ref{KSinL} is equal to $K(a_k)$.
\end{proof}

Finally, some small corrections are needed in the proof of
Proposition~5.12 of~\cite{[K1]}:

\begin{proposition}                         \label{tame}
Assume that $(K,v)$ is a henselian field. Then $(K,v)$ is a tame field
if and only if for every element $x$ in any extension $(L,v)$ of $(K,v)$
there exists a weakly pure homogeneous sequence for $x$ over $K$,
provided that $x$ is transcendental over $K$.
\end{proposition}
\begin{proof}
First, let us assume that $(K,v)$ is a tame field and that $x$ is an
element in some extension $(L,v)$ of $(K,v)$, transcendental over $K$.
We set $a_0=0$. We assume that $k\geq 0$ and that $a_i$ for $i\leq k$
are already constructed. Like $K$, also the finite extension $K_k:=
K(a_0,\ldots,a_k)$ is a tame field. Therefore, if $x$ is the pseudo
limit of a pseudo Cauchy sequence in $K_k\,$, then this pseudo Cauchy
sequence must be of transcendental type, and $K_k(x)|K_k$ is pure and
hence weakly pure in $x$.

If $K_k(x)|K_k$ is weakly pure in $x$, then we take $a_k$ to be the last
element of ${\eu S}$ if $k>0$, and ${\eu S}$ to be empty if $k=0$.

Assume that this is not the case. Then $x$ cannot be the pseudo limit of
a pseudo Cauchy sequence without pseudo limit in $K_k\,$. So the set
$v(x-a_k- K_k)$ must have a maximum, say $x-a_k-d$ with $d\in K_k$.
Since we assume that $K_k(x)|K_k$ is not weakly pure in $x$, there exist
${\rm e}\in\N$ and $c\in K_k$ such that $vc(x-a_k-d)^{\rm e}=0$ and
$c(x-a_k-d)^{\rm e}v$ is algebraic over $K_kv$. Since $K_k$ is a tame
field, its residue field is perfect, so $c(x-a_k-d)^{\rm e}v$ is
separable-algebraic over $K_kv$. Also, if $\chara Kv=p>0$, then $vK_k$
is $p$-divisible and therefore, e can be chosen to be prime to $p$.
Since $v(x-a_k-d)$ is maximal in $v(x-a_k- K_k)$, we must have that
${\rm e}>1$ or $c(x-a_k-d)^{\rm e}v\notin K_kv$.

Now Lemma~\ref{exkrap} shows that there exists $a\in \tilde{K}$,
strongly homogeneous over $K_k$ and such that $v(x-a_k-d-a)>v(x-a_k-d)$.
So $a+d$ is a homogeneous approximation of $x-a_k$ over $K_k$, and we
set $a_{k+1}:=a_k+a+d$. This completes our induction step. If our
construction stops at some $k$, then $K_k(x)|K_k$ is weakly pure in $x$
and we have obtained a weakly pure homogeneous sequence. If the
construction does not stop, then $S=\N$ and the obtained sequence is
pure homogeneous.

\pars
For the converse, assume that $(K,v)$ is not a tame field. We
choose an element $b\in\tilde{K}$ such that $K(b)|K$ is not a tame
extension. On $K(b,x)$ we take the valuation $v_{b,\gamma}$ with
$\gamma$ an element in some ordered abelian group extension such that
$\gamma>vK$. Choose any extension of $v$ to $\tilde{K}(x)$. Since $vK$
is cofinal in $v\tilde{K}$, we have that $\gamma>v\tilde{K}$. Since
$b\in \tilde{K}$, we find $\gamma\in v\tilde{K}(x)$. Hence,
$(\tilde{K}(x)|\tilde{K},v)$ is value-transcendental.

Now suppose that there exists a weakly pure homogeneous sequence
${\eu S}$ for $x$ over $K$. By Lemma~3.3 of~\cite{[K1]}, also
$(K_{\eu S}(x)|K_{\eu S},v)$ is value-transcendental. Since $(K_{\eu S}(x)
|K_{\eu S}, v)$ is also weakly pure, it follows that there must be some
$c\in K_{\eu S}$ such that $x-c$ is a value-transcendental element (all
other cases in the definition of ``weakly pure'' lead to immediate or
residue-transcendental extensions). But if $c\ne b$ then $v(b-c)\in
v\tilde{K}$ and thus, $v(c-b)<\gamma$. This implies
\[
v(x-c)=\min\{v(x-b),v(b-c)\}=v(b-c) \in v\tilde{K}\>,
\]
a contradiction. This shows that $b=c\in K_{\eu S}$. On the other hand,
$K_{\eu S}$ is a tame extension of $K$ by Proposition~\ref{hit} and
cannot contain $b$. This contradiction shows that there cannot exist a
weakly pure homogeneous sequence for $x$ over~$K$.
\end{proof}

%
%
\section{Other corrections for the paper \cite{[K1]}}  \label{sectcor}
\noindent
$\bullet$ \ In the paragraph before Example 3.9, ``relatively closed
subfield'' should be: ``relatively algebraically closed subfield''.

\sn
$\bullet$ \ The sentence after the first display in Lemma 3.13 should
read: ``Then $K(a)\subseteq \ic (K(x)|K,v)$''.

\sn
$\bullet$ \ In the third line of Theorem~6.1, ``$\Gamma$'' should be
replaced by ``$G$''.


\sn
$\bullet$ \ In the proof of Theorem 6.1, ``for the proof of assertion a)
it now suffices...'' (line after the fifth display) should be replaced
by: ``for the proof of assertions a) and b) it now suffices...''.

 \bn\bn


%
%
\section{Proofs of Theorems~\ref{gen1.1}
and~\ref{tame_rat_ff}}                      \label{sectproofs}
\noindent
{\bf Proof of Theorem~\ref{gen1.1}:}
\sn
%
Since $\Gamma / vL$ is a finite group, $k|Lv$ is a finite extension and
$v$ is nontrivial on $L$, from Theorem~2.14 of~\cite{[K1]} it follows
that there is a separable-algebraic extension $(L(a),v)$ of $(L,v)$ such
that $vL(a)=\Gamma$ and $L(a)v=k$. Then $L(a)$ is an infinite
separable-algebraic extension of $K^h$. Therefore, without loss of
generality we can assume that $vL=\Gamma $ and $Lv=k$.

Since $L|K^h$ is an infinite separable-algebraic extension, from
Theorem~\ref{MTai} it follows that $(L,v)$ admits an immediate
extension $(M,v)$ of infinite transcendence degree. Take elements
$x_1,\ldots, x_{n-1},y\in M$ algebraically independent over $L$ and set
\[
E\>:=\> K(x_1,\ldots, x_{n-1})\subseteq M.
\]
As $(L(x_1,\ldots, x_{n-1})|L,v)$ is an immediate extension, we obtain that
\[
vE\> \subseteq\> vL(x_1,\ldots, x_{n-1})\>=\> \Gamma
\textrm{ \, and \, }
Ev\> \subseteq  \> L(x_1,\ldots, x_{n-1})v\> =\> k.
\]
Since $L|K^h$ and $K^h|K$ are separable-algebraic extensions, also $L|K$
is separable-algebraic. Furthermore $vL/vK$ is a finite group and the
extension $Lv|Kv$ is finite, hence there is a finite subextension $L'|K$
of $L|K$ such that $vL'=vL$ and \mbox{$L'v=Lv$.} Moreover, by the
Theorem of Primitive Element, we can choose $L'|K$ to be a simple
extension $K(b)|K$ for some $b\in L$. Then $E(b)\subseteq L(x_1,\ldots,
x_{n-1})$, hence $vE(b)=vL=\Gamma $ and \mbox{$E(b)v=Lv=k$.}

Multiplying $y$ by an element in $K^{\times}$ of large enough value if
necessary, we can assume that
\[
vy\> > \> \textrm{kras}(b,E)\in \widetilde{vE}=\widetilde{vK}.
\]
Since
\[
E(b)\>\subseteq\> E(y,b)\>\subseteq\> L(x_1,\ldots, x_{n-1},y)
\>\subseteq\> M\>,
\]
the extensions $(L(x_1,\ldots, x_{n-1},y)|E(b),v)$ and $(E(y,b)|E(b),v)$
are immediate. Take a transcendental element $x_n$ in some field
extension of $E$ and define by $y\mapsto x_n-b$ an isomorphism of
$E(y,b)$ onto $E(x_n,b)$. This isomorphism induces a valuation $w$ on
$E(x_n,b)$, which is an extension of the valuation $v$ of $E(b)$ with
$w(x_n-b)=vy$. Hence, $w(x_n-b)>\,$kras$(b,E)$ and from Lemma~3.13
of~\cite{[K1]} we deduce that
\begin{eqnarray*}
vL\>=\>vE(b)\>\subseteq\> wE(x_n) \>=\> vE(y+b) \>\subseteq \>
vL(x_1,\ldots, x_{n-1},y)\>=\>vL,\\
vL\>=\>E(b)v\>\subseteq\> E(x_n)w \>=\> E(y+b)v \>\subseteq \>
L(x_1,\ldots, x_{n-1},y)v\>=\>Lv,
\end{eqnarray*}
since $(L(x_1,\ldots, x_{n-1},y)|L,v)$ is an immediate extension.
Thus equality holds everywhere and $w$ is an extension of $v$ from $K$
to $K(x_1,\ldots, x_n)$ such that
\[
wK(x_1,\ldots,x_n)\>=\>wE(x_n)\>=\>vL\>=\>\Gamma  \textrm{ \ and \ }
 K(x_1,\ldots,x_n)w\>=\>E(x_n)w\>=\>Lw\>=\>k.
\]
\QED

\parm
For the proof of Theorem~\ref{tame_rat_ff} we will need the following
lemma from~\cite{[B]}.

\begin{lemma}                            \label{rat_ff_IC}
Assume that
$vF/vK$ is a torsion group and $Fv|Kv$ is an algebraic extension. Fix an
extension of $v$ to $\widetilde{F}$ and set $L:=IC(F|K,v)$. If the
order of each element of $vF/vK$ is prime to the characteristic exponent
of $Kv$ and $Fv|Kv$ is separable, then $vL=vF$ and $Lv=Fv$, and the
extension $(L(x_1,\ldots,x_n)|L,v)$ is immediate.
\end{lemma}

\sn
\bfind{Proof of Theorem \ref{tame_rat_ff}:}
Assume that at least one of the two extensions $\Gamma|vK$ and $k|Kv$ is
infinite or $(K,v)$ admits an immediate extension of transcendence
degree $n$. Then parts A2) and B4) of Theorem 1.6 of~\cite{[K1]} state
that in both cases the valuation $v$ admits an extension to $F$ such
that $vF = \Gamma$ and $Fv = k$.

Assume now that there is an extension of $v$ to $F$ with $vF = \Gamma$
and $Fv = k$. Fix an extension of this valuation to $\widetilde{F}$ and
denote it again by $v$. Take $K^h$ and $F^h$ to be the henselizations of
$K$ and $F$ with respect to this extension. Set $L:= IC(F|K,v)$. Then
$L$ is a separable-algebraic extension of $K$ which contains $K^h$. As
$vK$ is $p$-divisible and $Kv$ is perfect by assumption, the order of
each element of $\Gamma /vK$ is prime to $p$ and $k|Kv$ is a
separable-algebraic extension. Hence, Lemma~\ref{rat_ff_IC} yields that
$(L(x_1,\ldots,x_n)|L,v)$ is an immediate extension with $vL=\Gamma$ and
$Lv=k$. Moreover, if $\Gamma/vK=vL/vK^h$ or $k|Kv=Lv| K^hv$ is infinite,
then by the fundamental inequality, also the extension $L|K^h$ is
infinite.

Suppose that the extensions $\Gamma|vK$ and $k|Kv$ are finite. Assume
first that $L|K^h$ is an infinite extension. Take a finite subextension
$E|K^h$ of degree bigger than $(\Gamma :vK)[k:Kv]$. Then
\[
[E:K^h]\> >\>(\Gamma :vK)[k:Kv]\>=\>(vL:vK^h)[Lv:K^hv]\> \geq \>
(vE:vK^h)[Ev:K^hv],
\]
and thus the extension $(E|K^h,v)$ has a nontrivial defect. In this
case, or if $L|K^h$ is itself a finite defect extension, we have that
$p>1$ and part 2) of Theorem~\ref{almK_inftrdeg} yields that every
maximal immediate extension of $(K^h,v)$ is of infinite transcendence
degree. Thus the same holds for $(K,v)$ and in particular, $(K,v)$
admits an immediate extension of transcendence degree $n$.

It remains to consider the case of $(L|K^h,v)$ finite and defectless.
%
%
As the extension $(L(x_1,\ldots,x_n)|L,v)$ is immediate, it is
contained in a maximal immediate extension $(M|L,v)$. If there is a
maximal immediate extension of finite transcendence degree over $L$,
then by part 1) of Theorem~\ref{almK_inftrdeg} it is isomorphic to $M$
over $L$. This shows that every maximal immediate extension is of
transcendence degree at least $n$ over~$L$.

Take a maximal immediate extension $(M,w)$ of $(K^h,v)$. Take the unique
extension of the valuation $w$ of $M$ to $M.L$ and call it again $w$.
Since $K^h$ is henselian, the restriction of $w$ to $L$ coincides with
$v$. As $L|K^h$ is a finite defectless extension of henselian fields, by
Lemma~2.5 of~\cite{[K2]} it is linearly disjoint from $M|K^h$ and the
extension $(M.L|L,w)$ is immediate. Since a finite extension of a
maximal field is again maximal, the field $(M.L,w)$ is a maximal
immediate extension of $(L,v)$. As we have already shown,
trdeg$M.L|L\geq n$. Hence also trdeg $M|K^h\geq n$. Since $(M,v)$ is
also a maximal immediate extension of $(K,v)$, we deduce that $K$ admits
an immediate extension of transcendence de gree $n$.
\QED

\newcommand{\lit}[1]{\bibitem #1{#1}}

\end{document}